%% file: segmentieng.tex
\documentclass[b5paper]{article}
\usepackage[english]{babel}
\input com_seg
\newenvironment{dimo}{\begin{proof}[{\rm\bf Proof}]}{\end{proof}}
\newenvironment{dimono}{\noindent{\rm\bf Proof.}}{}
\def\figura{\normalsize\stepcounter{figura}Figure \thefigura.}
\newtheorem*{main}{Theorem (T1)}
\newtheorem*{prop}{\stepcounter{proposizione}Proposition \theproposizione{} (P\theproposizione)}
\newtheorem*{corol}{\stepcounter{corollario}Corollary \thecorollario}
\newtheorem*{pro}{\stepcounter{proposizione}Proposition \theproposizione}
\begin{document}
\begin{center}
  {\huge\bf Parallel Line Segments}\\[5mm]
  {\small author: \bf Antonio Polo}\\
  {\small Italian High School of Rovinj (Croatia),}\\
  {\small University ``Juraj Dobrila'' of Pula (Croatia),\\
    email: {\sf toni.rovigno@gmail.com}}\\[5mm]
  {\bf Abstract}\\[2mm]
  \begin{minipage}{0.8\linewidth}
    In this article I will address some questions about a mathematical problem that my friend Patrizio Frederic, a researcher in
    statistics at the University of Modena, proposed to me. Given some parallel line segments, is there at least one straight line that
    passes through all of them? If there were many lines that solve the problem, can I choose a ``best one'' among all of them? I will
    fully address the first question.  As for the second question, I will illustrate it with some ``experimental'' examples and suggest an
    outline for future explorations.
  \end{minipage}\\[3mm]
  \begin{minipage}{0.75\linewidth}
  \end{minipage}\\
  \input{figura00}
\end{center}
\subsec{Introduction}
When the problem have been formulated for the first time to us, the question was exactly this: «Do you know if there exists a quick way to
determine whether we can intersect a series of parallel line segments with a straight line?». In this article, that we have divided into
five sections, we have collected everything we have found on this subject. The first section includes this introduction. The second and the
third are sections of technical nature in which we define the tools that we use, the real functional $\Phi$ and the space of real affine
functions \LL, and we prove their main properties. The fourth is the section in which we prove the main theorem, and we give a
comprehensive answer to the question above. In the last chapter we studied some practical examples, applying the results obtained and also
trying to give some hints to answer a question that have arisen in the meantime in us: «If the straight lines that solve the problem are
many, with what criteria can I choose one of them saying that it is better than the others?».
\subsec{A particular operator}
We begin by defining a tool that allows us to determine which are the respective positions of three points $A(x_A,y_A)$, $B(x_B,y_B)$ and
$C(x_C,y_C)$ on the Cartesian coordinate system. In particular we want to know if the three points are colinear or, alternatively, on what
half-plane is one of the three points with respect to the straight line through the other two. This tool will be a functional defined by
the determinant of a special square matrix of order 3.

For every given three points $A(x_A,y_A)$, $B(x_B,y_B)$ and $C(x_C,y_C)$ on the Cartesian coordinate system, the functional
$\Phi:(\R^2)^3\to\R$ is defined as follows: $$\Phi(A,B,C):=\det\begin{pmatrix}1&1&1\\x_A&x_B&x_C\\y_A&y_B&y_C\end{pmatrix}.$$
We use the multilinearity of the determinant to prove some properties of the functional $\Phi$.
\begin{prop}
  $\Phi$ is invariant for translations.
\end{prop}
\begin{dimono}
  Let $\mathcal T:\R2\to\R2$ be a translation given by a vector $\vec v=(a,b)$. We have that
  \begin{gather*}
    \Phi\left(\mathcal T(A),\mathcal T(B),\mathcal T(C)\right)=
    \det\begin{pmatrix}1&1&1\\x_A+a&x_B+a&x_C+a\\y_A+b&y_B+b&y_C+b\end{pmatrix}=
  \end{gather*}
  \begin{gather*}
    =\det\begin{pmatrix}1&1&1\\x_A&x_B&x_C\\y_A+b&y_B+b&y_C+b\end{pmatrix}+
    a\cdot\overbrace{\det\begin{pmatrix}1&1&1\\1&1&1\\y_A+b&y_B+b&y_C+b\end{pmatrix}}^{=0}=\\
    =\det\begin{pmatrix}1&1&1\\x_A&x_B&x_C\\y_A&y_B&y_C\end{pmatrix}+
    b\cdot\overbrace{\det\begin{pmatrix}1&1&1\\x_A&x_B&x_C\\1&1&1\end{pmatrix}}^{=0}=\Phi(A,B,C)\,.\tag*{$\square$}
  \end{gather*}
\end{dimono}
\begin{prop}
  $\Phi$ is invariant for rotations.
\end{prop}
\begin{dimono}
  Let $\EuScript R:\R2\to\R2$ be a rotation with rotation angle $\alpha$. We have that
  \begin{gather*}
    \Phi\left(\EuScript R(A),\EuScript R(B),\EuScript R(C)\right)=
    \det\left[\begin{pmatrix}1&0&0\\0&\cos\alpha&-\sin\alpha\\0&\sin\alpha&\cos\alpha\end{pmatrix}\cdot
      \begin{pmatrix}1&1&1\\x_A&x_B&x_C\\y_A&y_B&y_C\end{pmatrix}\right]=\\
    =\overbrace{\det\begin{pmatrix}1&0&0\\0&\cos\alpha&-\sin\alpha\\0&\sin\alpha&\cos\alpha\end{pmatrix}}^{=1}\cdot
    \det\begin{pmatrix}1&1&1\\x_A&x_B&x_C\\y_A&y_B&y_C\end{pmatrix}=\Phi(A,B,C)\,.\tag*{$\square$}
  \end{gather*}
\end{dimono}
\begin{prop}
  Three points $A$, $B$ and $C$ on the Cartesian coordinate system are colinear iff $\Phi(A,B,C)=0$.
\end{prop}
\begin{dimo}
  If two of the three points coincide then the proposition is obvious, therefore we can assume that the three points are distinct.
  For the proposition P2 is not restrictive to assume also that the points $A$ and $B$ are colinear with the $x$-axis, i.e. $y_A=y_B$. We
  have that $$\Phi(A,B,C)=\det\begin{pmatrix}1&1&1\\x_A&x_B&x_C\\y_A&y_A&y_C\end{pmatrix}=(x_A-x_B)(y_A-y_C)\;.$$
  For the assumption that the points are distinct we have that $x_A\ne x_B$, then $\Phi(A,B,C)=0$ iff $y_A=y_C$, i.e. iff the points $A$,
  $B$ and $C$ are colinear.
\end{dimo}
\begin{prop}
  Let $A$, $B$ and $C$ be three points on the Cartesian plane, distinct and not colinear, and let $\mathscr C$ be the circle circumscribed
  to the triangle $ABC$, then $\Phi(A,B,C)>0$ iff, along $\mathscr C$, the points $A$, $B$ and $C$ are ordered counterclockwise.
  Conversely $\Phi (A,B,C)<0$ iff along $\mathscr C$ the points $A$, $B$ and $C$ are ordered clockwise.
\end{prop}
\begin{dimono}
  For P1 is not restrictive to assume that the circumcenter of the triangle $ABC$ coincides with the origin of the Cartesian plane and
  for P2 is not restrictive to assume that the points $A$ and $B$ are aligned with the $y$-axis , so that $x_A=x_B$ and $y_A<y_B$. Let $R$
  be the radius of the circle circumscribed to the triangle $ABC$, then for the above assumptions, there are two angles $\alpha,\gamma$
  satisfying $0<\alpha<\pi$, $0\le\gamma<2\pi$, $\gamma\ne\pm\alpha$ and such that $A(R\cos\alpha,-R\sin\alpha)$,
  $B(R\cos\alpha,R\sin\alpha)$ and $C(R\cos\gamma,R\sin\gamma)$ are the Cartesian coordinates of the points $A$, $B$ and $C$.
  \begin{center}
    \input{figura01}
  \end{center}
  In Figures 1 and 2 we highlight the fact that if the points $A$, $B$ and $C$ are ordered counterclockwise, then
  $\cos\alpha>\cos\gamma$, while if the points $A$, $B$ and $C$ are ordered clockwise, then $\cos\alpha<\cos\gamma$. But by calculating 
  $\Phi(A,B,C)$ we have that
  \begin{gather*}
    \Phi(A,B,C)=\det\begin{pmatrix}1&1&1\\R\cos\alpha&R\cos\alpha&R\cos\gamma\\
      -R\sin\alpha&R\sin\alpha&R\sin\gamma\end{pmatrix}=2R^2\sin\alpha(\cos\alpha-\cos\gamma)
  \end{gather*}
  and since $2R^2\sin\alpha>0$, we conclude that:
  \begin{align*}
    &\Phi(A,B,C)>0\!\quad\Leftrightarrow\!\quad\cos\alpha>\cos\gamma\!\quad\Leftrightarrow\!\quad\text{the orientation is
    counterclockwise,}\\
    &\Phi(A,B,C)<0\!\quad\Leftrightarrow\!\quad\cos\alpha<\cos\gamma\!\quad\Leftrightarrow\!\quad\text{the orientation is
    clockwise.}\tag*{$\square$}
  \end{align*}
\end{dimono}%
To conclude this section, always regarding the functional $\Phi$, we make one final point that will be useful below.
\begin{prop}
  Let $A$, $B$, $C$ and $D$ be four points given on the Cartesian plane, positioned so that $x_B<x_C$ and $x_A=x_D$. The following
  conditions are equivalent:
  \begin{enumerate}[\rm i)]
  \item $y_A\ge y_D$,
  \item $\Phi(A,B,C)\ge\Phi(D,B,C)$,
  \item $\Phi(B,A,C)\le\Phi(B,D,C)$,
  \item $\Phi(B,C,A)\ge\Phi(B,C,D)$.
  \end{enumerate}
  Moreover in all four cases equality holds iff $A\equiv D$.
\end{prop}
\begin{dimo}
  ($\bm{{\rm i}\Leftrightarrow {\rm ii}}$) We compute:
  \begin{gather*}
    \Phi(A,B,C)-\Phi(D,B,C)=\det\begin{pmatrix}1&1&1\\x_A&x_B&x_C\\y_A&y_A&y_C\end{pmatrix}-
    \det\begin{pmatrix}1&1&1\\x_A&x_B&x_C\\y_D&y_A&y_C\end{pmatrix}=\\
    =\det\begin{pmatrix}0&1&1\\0&x_B&x_C\\y_A-y_D&y_A&y_C\end{pmatrix}=(y_A-y_D)(x_C-x_B)\;,
  \end{gather*}
  by assumption $x_C-x_B>0$, thus $\Phi(A,B,C)\ge\Phi(D,B,C)\Leftrightarrow y_A\ge y_D$, and moreover
  $\Phi(A,B,C)=\Phi(D,B,C)\Leftrightarrow y_A=y_D\Leftrightarrow A\equiv D$.\\[1mm]
  ($\bm{{\rm ii}\Leftrightarrow{\rm iii}\Leftrightarrow{\rm iv}}$) These equivalences follow from the fact that the determinant is
  alternating (or skew-symmetric) as well as multilinear.
\end{dimo}
\subsec{The space of affine functions}
Another tool we use is the space \LL, wherein we represent the real affine functions, i.e. the functions $f:\R\to\R$ such that $f(x)=mx+b$ 
for some $(m,b)\in\R2$. 

The elements of \LL are in one-to-one correspondence with the straight lines on the Cartesian plane that are not parallel to the
$y$-axis, these are in one-to-one correspondence with the equations of the form $y=mx+b$ and then, by the pair $(m,b)$, they are in
one-to-one correspondence with the pairs of \R2. This one-to-one correspondence has some interesting properties as shown in the next
proposition.

We use also the concept of ``sheaf of straight lines'', that is the set of straight lines all passing through the same point (except almost
that parallel to the $y$-axis) or all parallel to each other.
\begin{pro}
  Let \LL be the space of affine real functions. In \LL we have that the following statements occur.\\[3pt]
  {\bf(L1)} The set of all sheaves of straight lines on the Cartesian plane is in one-to-one correspondence with the set of straight lines
  of  \LL and, in particular, a straight line in \LL is parallel to the $y$-axis if and only if its associated sheaf is composed by parallel
  straight lines.\\[3pt]
  {\bf(L2)} The set of straight lines passing through a line segment parallel to the $y$-axis is convex in \LL and, in particular, it
  consists of  the points of the plane between two parallel straight lines.\\[3pt] 
  {\bf(L3)} The set of straight lines passing through the points of two line segments parallel to the $y$-axis, but with distinct 
  $x$-coordinate, is limited and convex in \LL and, in particular, it is a parallelogram whose barycenter is associated with the straight
  line passing through the midpoints of the two line segments.
\end{pro}
\begin{dimo}
  {\bf(L1)} Let $P(x_P,y_P)\in\R2$ be a generic point, then the sheaf of straight lines passing through $P$ has explicit equation
  $y=m(x-x_P)+y_P$ and for all different $m\in\R$, in \LL the sheaf is represented by the locus of points whose coordinates are
  $(m,-x_Pm+y_P)$. This locus is a straight line whose explicit equation is $y=-x_Px+y_P$, then it is not parallel to the $y$-axis. However,
  if we have a sheaf of parallel straight lines in \R2, we can express its equation in the form $y=kx+b$, where $b$ is a parameter and $k$
  is constant. For all $b\in\R$ this sheaf is represented in \LL by the locus of points whose coordinates are $(k,b)$. Such a locus
  is a straight line, parallel to the $y$-axis, whose equation is $x=k$.
  \begin{center}
    \input{figura02}\hfill$\Longleftrightarrow$\hfill\input{figura03}
  \end{center}
  {\bf(L2)} Let $\overline{AB}$ be a line segment parallel to the $y$-axis whose end-points are $A(a,b)$ and $B(a,c)$ (Figure 3). A generic
  point $P_t$ of this segment has coordinates of type $P_t\big(a,b+t(c-b)\big)$ for some $t\in[0,1]$. As seen in (L1), the sheaf of
  straight lines passing through $P_t$ is represented in \LL by the straight lines with equation $y=-ax+b+t(c-b)$, for all
  $t\in[0,1]$ this straight lines represent the set of all parallel lines between the lines $y=-ax+b$ and $y=-ax+c$ and then they form a
  convex set in \LL.\\[2pt] 
  {\bf(L3)} For L2 it is clear that the set considered is a parallelogram.~Then let $A(a,b)$, $B(a,c)$, $C(d,e)$ and $D(d,f)$ be the
  end-points of the two line segments and let $E\big(a,\frac{b+c}2\big)$ and $F\big(d,\frac{e+f}2\big)$ be the midpoints of
  $\overline{AB}$ and $\overline{CD}$, respectively. The equation of the straight line $EF$ will be
$$y=\frac{b+c-e-f}{2(a-d)}(x-d)+\frac{e+f}2=\frac{b+c-e-f}{2(a-d)}x+\frac{ae+af-bd-cd}{2(a-d)}\,.$$
  The equations of the straight lines $AC$, $AD$, $BC$ and $BD$ are:
  \begin{align*}
    AC:&\,y=\frac{b-e}{a-d}x+\frac{ae-bd}{a-d}\quad,&AD:&\,y=\frac{b-f}{a-d}x+\frac{af-bd}{a-d}\;,\\
    BC:&\,y=\frac{c-e}{a-d}x+\frac{ae-cd}{a-d}\quad\text{ e}&BD:&\,y=\frac{c-f}{a-d}x+\frac{af-cd}{a-d}\;.
  \end{align*}
  These four straight lines, in \LL, represent the vertices of the parallelogram corresponding to all straight lines passing through $AB$
and
  $CD$. The barycenter of this parallelogram is the midpoint of one of its diagonals, which is the point 
  $$P\left(\frac{\frac{b-e}{a-d}+\frac{c-f}{a-d}}{2},\frac{\frac{ae-bd}{a-d}+\frac{af-cd}{a-d}}{2}\right)\;
  \Rightarrow\;P\left(\frac{b-e+c-f}{ 2(a-d)},\frac{ae-bd+af-cd}{2(a-d)}\right)$$ whose coordinates are exactly the slope and the
  $y$-intercept of the straight line $EF$.
\end{dimo}
\subsec{The problem}
The problem we want to solve is the following: given a finite number, $n\ge2$, of line segments which are closed, parallel, but
mutually non-aligned, find necessary and sufficient conditions for having at least one straight line that intersects all
segments.\vskip3mm\noindent
{\bf Notations} (Figure 4){\bf.}
\begin{itemize}
  \item Let the line segments be on a Cartesian coordinate system, so that they are parallel to the $y$-axis, and let's call them
   $L_1,L_2,\ldots,L_n$ in ascending order according to the $x$-axis.
  \item Let $A_i(x_i,a_i)$ and $B_i(x_i,b_i)$ be the end-points of the $i$-th line segment, with $a_i\le b_i$ for all $i\in\{1,2,\dots,n\}$
   and $x_1<x_2<\ldots<x_n$.
  \item Among all the straight lines passing through the end-points of the line segments we fix two particular ones: 
  the straight line $A_sB_t$ whose slope $m_{st}$ is the minimum between the slopes of all the straight lines $A_iB_j$ such that 
  $1\le i<j\le n$  and the straight line $B_uA_v$ whose slope $m'_{uv}$ is the maximum between the slopes of all the straight lines
  $B_iA_j$ such that $1\le i<j\le n$.
\end{itemize}
\begin{center}
  \input{figura04}
\end{center}
\begin{main}
  With the previous notations and considering the non-trivial case $n\ge3$, we have that the following statements are equivalent.
  \begin{enumerate}[\rm i)]
  \item At least one straight line intersects all line segments $L_1,L_2,\ldots,L_n$.
  \item $\Phi(A_i,B_j,A_k)\le0\le\Phi(B_i,A_j,B_k)$ for all $i,j,k$ such that $1\!\le \!i<\!j\!<\!k\!\le \!n$.
  \item $\Phi(A_s,B_t,A_j)\le0\le\Phi(B_i,A_s,B_t)$ for all $i,j$ such that $1\le i<s$ and $t<j\le n$.
  \item $\Phi(A_i,B_u,A_v)\le0\le\Phi(B_u,A_v,B_j)$ for all $i,j$ such that $1\le i<u$ and $v<j\le n$.
  \item The straight line $A_sB_t$ intersects all line segments $L_1,L_2,\dots,L_n$.
  \item The straight line $B_uA_v$ intersects all line segments $L_1,L_2,\dots,L_n$.
  \end{enumerate}
\end{main}
\begin{proof}[{\rm \bf Proof}]
  ($\bm{{\rm i}\Rightarrow{\rm ii}}$) We fix three integers $i,j,k$ such that $1\le i<j<k\le n$, then for assumption there exist three
  colinear points $P_i(x_i,y_i)\in L_i$, $P_j(x_j,y_j)\in L_j$ and $P_k(x_k,y_k)\in L_k$ such that $a_p\le y_p\le b_p$ for all
  $p\in\{i,j,k\}$. Then for P3 and P5 we have that\vskip-5mm
  \begin{gather*}
    \Phi(A_i,B_j,A_k)\le\Phi(P_i,B_j,A_k)\le\Phi(P_i,P_j,A_k)\le\overbrace{\Phi(P_i,P_j,P_k)}^{=0}\le\\
    \le\Phi(B_i,P_j,P_k)\le\Phi(B_i,A_j,P_k)\le\Phi(B_i,A_j,P_k)\,.
  \end{gather*}
  ($\bm{{\rm ii}\Rightarrow{\rm iii}}$), ($\bm{{\rm ii}\Rightarrow{\rm iv}}$). Obvious.\\[1mm]
  ($\bm{{\rm iii}\Rightarrow{\rm v}}$) For all $i\in\{1,2,\ldots,n\}$ let $P_i(x_i,y_i)$ be the intersection point between the straight
  lines $A_sB_t$ and $x=x_i$. We analyze four cases.\\[.5mm]
  $1^{\rm st}$ case: $1\le i<s$. For assumption $\Phi(B_i,A_s,B_t)\!\ge \!0$ and for P3 $\Phi(P_i,A_s,B_t)\!=0$, then for P5 we have that
  $b_i\ge y_i$.\\[.5mm]
  $2^{\rm nd}$ case: $1\le i<t$. We know that $m_{st}\le m_{it}$ and that the point $P_i$ belongs to the straight line $A_sB_t$, then we
  have that $$m_{st}\le m_{it}\quad\Leftrightarrow\quad\frac{y_i-b_t}{x_i-x_t}\le\frac{a_i-b_t}{x_i-x_t}
  \quad\Leftrightarrow\quad y_i-b_t\ge a_i-b_t\quad\Leftrightarrow\quad y_i\ge a_i\,.$$
  $3^{\rm rd}$ case: $s<i\le n$. We know that $m_{st}\le m_{si}$ and that the point $P_i$ belongs to the straight line $A_sB_t$, then we
  have that $$m_{st}\le m_{si}\quad\Leftrightarrow\quad\frac{a_s-y_i}{x_s-x_i}\le\frac{a_s-b_i}{x_s-x_i}
  \quad\Leftrightarrow\quad a_s-y_i\ge a_s-b_i\quad\Leftrightarrow\quad y_i\le b_i\,.$$
  $4^{\rm th}$ case: $t<i\le n$. For assumption $\Phi(A_s,B_t,A_i)\!\le\!0$ and for P3 $\Phi(A_s,B_t,P_i)\!=0$, then for P5 we have that
  $a_i\le y_i$.\\[.5mm]
  Adding the two trivial inequalities $b_s\ge y_s=a_s$ and $a_t\le y_t=b_t$, in the four cases are proved the inequalities
    $a_i\le y_i\le b_i$ for all $i\in\{1,2,\ldots n\}$, then the stright line $A_sB_t$ intersects all $n$ line segments.\\[1mm]
  ($\bm{{\rm iv}\Rightarrow{\rm vi}}$) Quite similar to the previous proof.\\[1mm]
  ($\bm{{\rm v}\Rightarrow{\rm i}}$), ($\bm{{\rm vi}\Rightarrow{\rm i}}$). Obvious.
\end{proof}
\begin{corol}
  The following statements are equivalent.
  \begin{enumerate}[\rm i)]
  \item There is one and only one straight line that intersects all the line segments $L_1,L_2,\dots,L_n$.
  \item The two straight lines $A_sB_t$ e $B_uA_v$ coincide and intersect all the line segments $L_1,L_2,\dots,L_n$.
  \item $\Phi(A_i,B_j,A_k)\le0\le\Phi(B_i,A_j,B_k)$ for all $i,j,k$ such that $1\le i<j<k\le n$ and, if we are not in presence of the limit
  case in which two of the segments degenerate into a point, then at least in one case there is the equality
  $$\Phi(A_i,B_j,A_k)=0\quad\text{or}\quad\Phi(B_i,A_j,B_k)=0\,.$$
  \end{enumerate}
\end{corol}
\begin{proof}[{\rm \bf Proof}]
  ($\bm{{\rm i}\Rightarrow{\rm ii}}$) Obvious for the equivalence ${\rm i}\Leftrightarrow{\rm iii}\Leftrightarrow{\rm iv}$ in T1.\\[1mm]
  ($\bm{{\rm ii}\Rightarrow{\rm iii}}$) for T1 we are sure that $\Phi(A_i,B_j,A_k)\le0\le\Phi(B_i,A_j,B_k)$ for all $i,j,k$ such that
  $1\le i<j<k\le n$. The four points $A_s$, $B_t$, $B_u$ e $A_v$ are, for assumption, colinear; assuming that there are not two segments
  that degenerate to a point, then it can not happen simultaneously that $A_s\equiv B_u$ and $B_t\equiv A_v$, then we are sure that in
  addition with the already known inequalities $ x_s <x_t$ e $x_u<x_v$, also at least one of these occurs:
  $$x_s<x_u\,,\quad x_u<x_s\,,\quad x_t<x_v\,,\quad x_v<x_t\,.$$ So, in each of the four cases we have that
  \begin{align*}
   \qquad\qquad x_s&<x_u\!\!\!\!\!&\text{and}&&\!\!\!\!\!x_u&<x_v&\Rightarrow&&\Phi(A_s,B_u,A_v)&=0\,,\qquad\qquad\\
   \qquad\qquad x_u&<x_s\!\!\!\!\!&\text{and}&&\!\!\!\!\!x_s&<x_t&\Rightarrow&&\Phi(B_u,A_s,B_t)&=0\,,\qquad\qquad\\
   \qquad\qquad x_t&<x_v\!\!\!\!\!&\text{and}&&\!\!\!\!\!x_s&<x_t&\Rightarrow&&\Phi(A_s,B_t,A_v)&=0\,,\qquad\qquad\\
   \qquad\qquad x_v&<x_t\!\!\!\!\!&\text{and}&&\!\!\!\!\!x_u&<x_v&\Rightarrow&&\Phi(B_u,A_v,B_t)&=0\,.\qquad\qquad
  \end{align*}
  ($\bm{{\rm iii}\Rightarrow{\rm i}}$) For T1 at least one straight line exists, the uniqueness would be obvious if two of the
  line segments were pointlike, so we assume that $\Phi(A_i,B_j,A_k)=0$ for some $i,j,k$ such that $1\le i<j<k\le n$ (the proof is similar
  if we assume that $\Phi(B_i,A_j,B_k)=0$). A straight line that crosses the segments $L_i$, $L_j$ and $L_k$ will pass through the three
  colinear points $P_i\in L_i$, $P_j\in L_j$ and $P_k\in L_k$, then for P5 we have that 
  $$0=\Phi(A_i,B_j,A_k)\le\Phi(P_i,B_j,A_k)\le\Phi(P_i,P_j,A_k)\le\Phi(P_i,P_j,P_k)=0\,.$$
  Since all inequalities are equalities then, once again for P5, we have that $P_i\equiv A_i$, $P_j\equiv B_j$ and $P_k\equiv A_k$  and the
  straight line can only be unique.
\end{proof}
\begin{corol}
  Among all the straight lines that pass through the line segments $L_1,L_2,\dots,L_n$ (assuming that there exist any), the straight lines
  $A_sB_t$  and $B_uA_v$ are those with maximum and minimum slope respectively.
\end{corol}
\begin{proof}[{\rm \bf Proof}]
  We consider a line $r$ with slope $m>m_{st}$, that passes through the points $P_s(x_s,y_s)\in L_s$ and $P_t(x_t,y_t)$ and we prove that 
  $P_t\notin L_t$. We calculate:
  \begin{gather*}
    m_{st}<m\quad\Rightarrow\quad \frac{a_s-b_t}{x_s-x_t}<\frac{y_s-y_t}{x_s-x_t}\quad\Rightarrow\\
    \Rightarrow\quad a_s-b_t>y_s-y_t\quad\Rightarrow\quad y_t-b_t>y_s-a_s\ge0\;.
  \end{gather*}
  Since $y_t>b_t$ then the straight line $r$ does not cross the line segment $L_t$ and this proves that $m_ {st}$ is the maximum slope
  between those of the straight lines that cross all segments. Similarly $m'_{uv}$ is the minimum of the slopes.
\end{proof}
\subsec{A special straight line}
Now we are able to suggest some algorithms that allow us to find a straight line, among others, that can cross in the ``best way'' all
line segments.\\[2mm]
{\bf Algorithm 1.} A (not so) special straight line.
\begin{enumerate}
 \item Let the $n$ line segments be given, then we find the straight line $r=A_sB_t$.
 \item We check if $\Phi(B_i,A_s,B_t)\ge0$ for all $i\in\{1,2,\ldots,s-1\}$.
 \item We check if $\Phi(A_s,B_t,A_j)\le0$ for all $j\in\{t+1,t+2,\ldots,n\}$.
 \item If the checks at points 2 and 3 have failed then there is no reason to continue, otherwise we can find also the straight line
  $p=B_uA_v$. For T1 we are sure that the straight lines $r$ and $p$ pass through all the line segments $L_1,L_2,\ldots,L_n$.
 \item We choose the line $s_1$ that, in \LL, corresponds to the midpoint between $r$ and $p$. For L2 the set of all the straight lines
  passing through all of the line segments $L_1,L_2,\ldots,L_n$ is convex, then $s_1$ also passes for all of those line segments.
\end{enumerate}
\begin{center}
  \input{figura05}\hfill\input{figura06}
\end{center}
In Figure 5 we have that
$$r\!:\,y=x+1\quad\text{and}\quad p\!:\,y=-\frac16x+\frac{43}6\quad\Rightarrow\quad s_1\!:\,y=\frac5{12}x+\frac{49}{12}\,;$$
in Figure 6 we have that
$$r\!:\,y=\frac 45x+\frac{14}5\quad\text{and}\quad p\!:\,y=-\frac13+\frac{22}3\quad\Rightarrow\quad
s_1\!:\,y=\frac{7}{30}x+\frac{76}{15}\,.$$
But this algorithm does not always give us a ``good'' straight line. As it is shown in Figure 6, sometimes the straight line we get does not
seem to be one of the best and in this example the straight line passes through the end-point of one of the line segments.\\[2mm]
{\bf Algorithm 2.} A straight line a little bit more special.
\begin{enumerate}
 \item We perform steps 1, 2, 3 and 4 of Algorithm 1.
 \item Among all the straight lines $A_iA_j$, with $1\le i<j\le n$, we take only those that cross all line segments
  $L_1,L_2,\ldots,L_n$, that are distinct from $r$ and $p$ and we take only once if two or more of them are identical.
 \item Among all the straight lines $B_iB_j$, with $1\le i<j\le n$, we take only those that cross all line segments
  $L_1,L_2,\ldots,L_n$, that are distinct from $r$ and $p$ and we take only once if two or more of them are identical.
 \item For L2 all the straight lines found in steps 2 and 3, together with $r$ and $p$,  are the vertices of a convex polygon in \LL. Let us
  take the straight line $s_2$ that corresponds to the discrete barycenter of this polygon.
\end{enumerate}
Let us review the example in Figure 5.
\begin{center}
  \input{figura07}\hfill\input{figura08}
\end{center}
In Figure 7 we have represented the line segments and in Figure 8 we have represented, in \LL, the convex polygon formed by all the
straight lines that cross all line segments; the point $G$ is the discrete barycenter of the polygon and it is associated with the
straight line $s_2$ that would seem to go through the line segments better than the straight line $s_1$ in Figure 5. In summary we have:
\renewcommand\minalignsep{2pt}
\begin{align*}
  A_4B_1\!:\;y=&-\frac16x+\frac{43}6&&\Rightarrow&&\!P_2\left(-\frac16,\frac{43}6\right),&
  A_2A_4\!:\;y=&\frac12x+\frac52\!&&\Rightarrow&&\!P_1\left(\frac12,\frac52\right),\\
  B_1B_4\!:\;y=&\frac13x+\frac{20}3\!&&\Rightarrow&&\!P_3\left(\frac13,\frac{20}3\right),&
  B_5A_2\!:\;y=&x+1\!&&\Rightarrow&&P_5\left(1,1\right),\\
  B_4B_5\!:\;y=&\frac12x+\frac{11}2&&\Rightarrow&&\!P_4\left(\frac12,\frac{11}2\right);
\end{align*}
$$m_G=\frac{-\frac16+\frac13+\frac12+\frac12+1}{5}=\frac{13}{30}\;,\quad
b_G=\frac{\frac{43}6+\frac{20}3+\frac{11}2+\frac52+1}{5}=\frac{137}{30}$$
and finally the equation of $s_2$ is $y=\dfrac{13}{30}x+\dfrac{137}{30}$.\vskip2mm\noindent
Let us review also the example in Figure 6.
\begin{center}
  \input{figura09}\hfill\input{figura10}
\end{center}
Even in this example the straight line $s_2$ crosses all the line segments in a very ``better'' way than the straight line $s_1$ do in
Figure 6. In summary in this example we have:
\begin{align*} 
  B_5A_3\!:\;y=&\frac45x+\frac{14}5\!&&\Rightarrow&&P_1\left(\frac45,\frac{14}5\right),&
  A_3B_1\!:\;y=&-\frac13x+\frac{22}3&&\Rightarrow&&\!P_2\left(-\frac13,\frac{22}3\right),\\[3mm]
  B_1B_4\!:\;y=&\frac13x+\frac{20}3\!&&\Rightarrow&&\!P_3\left(\frac13,\frac{20}3\right),&
  B_4B_5\!:\;y=&\frac12x+\frac{11}2&&\Rightarrow&&\!P_4\left(\frac12,\frac{11}2\right);
\end{align*}
$$m_G=\frac{\frac45-\frac13+\frac13+\frac12}4=\frac{13}{40}\;,\quad
b_G=\frac{\frac{14}5+\frac{22}3+\frac{20}3+\frac{11}2}4=\frac{223}{40}$$
and finally $s_2\!:\,y=\dfrac{13}{40}x+\dfrac{223}{40}$.\\[2mm]
{\bf Algorithm 3.} An even more special straight line.
\begin{enumerate}
  \item We perform steps 1, 2 and 3 of Algorithm 2.
  \item We do the same observations done in the step 4 of the algorithm 2, but in this case, as straight line $s_3$, we take the straight
  line associated with the continuous  barycenter of the polygon that we find in \LL.
\end{enumerate}
Let us review again the two examples above.
\begin{center}
  \input{figura11}\hfill\input{figura12}
\end{center}
We observe the polygon in Figure 8:
\begin{itemize}
 \item for L1 the straight lines $P_1P_2$, $P_2P_3$, $P_3P_4$, $P_4P_5$ and $P_5P_1$ have equations: 
  \begin{gather*}
    P_1P_2\!:\,y=-7x+6\;,\quad P_2P_3\!:\,y=-x+7\;,\quad P_3P_4\!:\,y=-7x+9\;,\\
    P_4P_5\!:\,y=-9x+10\;,\quad P_5P_1\!:\,y=-3x+4\;;
  \end{gather*}
 \item the area $A$ of the surface within the polygon is:
  \begin{align*}
    A&=\int_SdS=\int_{-\frac16}^{\frac13}\int^{-x+7}_{-7x+6}dydx+\int_{\frac13}^{\frac12}\!\!\int^{-7x+9}_{-7x+6}dydx+
    \int_{\frac12}^{1}\!\!\int^{-9x+10}_{-3x+4}dydx=\\
    &=\int_{-\frac16}^{\frac13}[y]^{-x+7}_{-7x+6}dx+\int_{\frac13}^{\frac12}[y]^{-7x+9}_{-7x+6}dx+
    \int_{\frac12}^{1}[y]^{-9x+10}_{-3x+4}dx=\\
    &=\int_{-\frac16}^{\frac13}(6x+1)dx+\int_{\frac13}^{\frac12}3dx+\int_{\frac12}^{1}(6-6x)dx=\\
    &=\big[3x^2+x\big]_{-\frac16}^{\frac13}+\big[3x\big]_{\frac13}^{\frac12}+
      \big[6x-3x^2\big]_{\frac12}^{1}=\frac34+\frac12+\frac34=\bm2\,;
  \end{align*}
 \item we will find the continuous barycenter $C(m_C,b_C) $ of the polygon following the calculations:
  \begin{align*}
    m_C&=\frac1A\int_SxdS=\\
    &=\frac12\left(\int_{-\frac16}^{\frac13}\int^{-x+7}_{-7x+6}\!\!\!xdydx+\int_{\frac13}^{\frac12}\!\!\!\int^{-7x+9}_{-7x+6}\!\!xdydx+
    \int_{\frac12}^{1}\!\!\int^{-9x+10}_{-3x+4}\!\!\!xdydx\right)=\\
    &=\frac12\left(\int_{-\frac16}^{\frac13}x[y]^{-x+7}_{-7x+6}dx+\int_{\frac13}^{\frac12}x[y]^{-7x+9}_{-7x+6}dx+
    \int_{\frac12}^{1}x[y]^{-9x+10}_{-3x+4}dx\right)=\\
    &=\int_{-\frac16}^{\frac13}\left(3x^2+\frac12x\right)dx+\int_{\frac13}^{\frac12}\frac32xdx+\int_{\frac12}^{1}(3x-3x^2)dx=\\
    &=\left[x^3+\frac14x^2\right]_{-\frac16}^{\frac13}\!\!+\left[\frac34x^2\right]_{\frac13}^{\frac12}\!\!+
      \left[\frac32x^2-x^3\right]_{\frac12}^{1}=\frac1{16}+\frac5{48}+\frac14=\bm{\frac5{12}}\,,
  \end{align*}
  \begin{align*}
    b_C&=\frac1A\int_SydS=\\
    &=\frac12\left(\int_{-\frac16}^{\frac13}\int^{-x+7}_{-7x+6}\!\!\!ydydx+\int_{\frac13}^{\frac12}\!\!\!\int^{-7x+9}_{-7x+6}\!\!ydydx+
    \int_{\frac12}^{1}\!\!\int^{-9x+10}_{-3x+4}\!\!\!ydydx\right)=\\
    &=\frac14\left(\int_{-\frac16}^{\frac13}\left[y^2\right]^{-x+7}_{-7x+6}dx+\int_{\frac13}^{\frac12}\left[y^2\right]^{-7x+9}_{-7x+6}dx+
    \int_{\frac12}^{1}\left[y^2\right]^{-9x+10}_{-3x+4}dx\right)=\\
    &=\!\!\!\int_{\!-\frac16}^{\frac13}\!\!\!\!\left(\!\frac{13}4\!+\!\frac{35}2x-12x^2\!\!\right)\!dx\!+\!\!\!
    \int_{\frac13}^{\frac12}\!\!\!\left(\!\frac{45}4\!-\!\frac{21}2x\!\!\right)\!dx\!+\!\!\!
    \int_{\frac12}^{1}\!\!\!\!\left(21-39x+18x^2\right)\!dx\!=\\
    &=\left[\frac{13}4x+\frac{35}4x^2-4x^3\right]_{-\frac16}^{\frac13}\!\!\!+
    \left[\frac{45}4x-\frac{21}4x^2\right]_{\frac13}^{\frac12}\!\!+
    \left[21x-\frac{39}2x^2+6x^3\right]_{\frac12}^{1}\!=\\
    &=\frac{35}{16}+\frac{55}{48}+\frac98=\bm{\frac{107}{24}}\;;
  \end{align*}
 \item we get the straight line $s_3$ with equation $y=\dfrac{5}{12}x+\dfrac{107}{24}$ that is associated with the point $C$ (Figure 11).
\end{itemize}
Now we observe the polygon in Figure 10:
\begin{itemize}
 \item for L1 the straight lines $P_1P_2$, $P_2P_3$, $P_3P_4$ and $P_4P_1$ have equations: 
  \begin{align*}
    P_1P_2\!:\,y&=-4x+6\;,&P_2P_3\!:\,y&=-x+7\;,\;,\\
    P_3P_4\!:\,y&=-7x+9\;,&P_4P_1\!:\,y&=-9x+10\;;
  \end{align*}
 \item the area $A$ of the surface within the polygon is:
  \begin{align*}
    A&=\int_SdS=\int_{-\frac13}^{\frac13}\int^{-x+7}_{-4x+6}\!dydx+\int_{\frac13}^{\frac12}\!\!\int^{-7x+9}_{-4x+6}\!dydx+
    \int_{\frac12}^{\frac45}\!\!\int^{-9x+10}_{-4x+6}\!dydx=\\
    &=\frac23+\frac7{24}+\frac9{40}=\bm{\frac{71}{60}}\,;
  \end{align*}
 \item we will find the continuous barycenter $C(m_C,b_C) $ of the polygon following the calculations:
  \begin{align*}
    m_C&=\frac1A\int_SxdS=\\
    &=\frac{60}{71}\left(\int_{-\frac13}^{\frac13}\!\int^{-x+7}_{-4x+6}\!\!xdydx+
    \int_{\frac13}^{\frac12}\!\!\int^{-7x+9}_{-4x+6}\!\!\!xdydx+\int_{\frac12}^{\frac45}\!\!\int^{-9x+10}_{-4x+6}\!\!xdydx\!\right)=\\
    &=\frac{60}{71}\left(\frac{2}{27}+\frac{13}{108}+\frac{27}{200}\right)=\bm{\frac{593}{2130}}\,,
  \end{align*}
  \begin{align*}
    b_C&=\frac1A\int_SydS=\\
    &=\frac{60}{71}\left(\int_{-\frac13}^{\frac13}\!\int^{-x+7}_{-4x+6}\!\!ydydx+
    \int_{\frac13}^{\frac12}\!\!\int^{-7x+9}_{-4x+6}\!\!\!ydydx+\int_{\frac12}^{\frac45}\!\!\int^{-9x+10}_{-4x+6}\!\!ydydx\!\right)=\\
    &=\frac{60}{71}\left(\frac{112}{27}+\frac{659}{432}+\frac{369}{400}\right)=\bm{\frac{11873}{2130}}\,;
  \end{align*}
 \item we get the straight line $s_3$ with equation $y=\dfrac{593}{2130}x+\dfrac{11873}{2130}$ that is associated with the point $C$
  (Figure 12).
\end{itemize}
In the two examples it seems clear that the straight line that we find with the first algorithm is not so ``good'', while it seems that
the two barycenters give us two straight lines quite similar. If this fact would be always true, then we might consider more appropriate to
use the algorithm 2, at least to avoid the huge amount of calculations that must be done to find the straight line with the algorithm 3.
But this result is accidental, due to the few examples seen, and the next example is the confirmation. I end this article proposing such an
example without calculations.\\[2mm]
In Figure 13 and 14 we have:
\begin{align*}
  A_1A_2\!:\,&y=x&&P_1\left(1,0\right)		&B_1B_4\!:\,&y=11&&P_5\left(0,11\right)\\
  A_2A_3\!:\,&y=\f12x+1&&P_2\left(\f12,1\right)&B_4B_5\!:\,&y=\f12x+8&&P_6\left(\f12,8\right)\\
  A_3A_6\!:\,&y=3&&P_3\left(0,3\right)		&B_5B_6\!:\,&y=x+4&&P_7\left(1,4\right)\\
  A_6B_1\!:\,&y=-x+12&&P_4\left(-1,12\right)	&B_6A_1\!:\,&y=\f32x-\f12&&P_8\left(\f32,-\f12\right)
\end{align*}
\begin{center}
  \input{figura13}\hfill\input{figura14}
\end{center}
$$s_1\!:\,y=\f14x+\f{23}4\;,\qquad s_2\!:\,y=\f7{16}x+\f{77}{16}\;,\qquad s_3\!:\,y=\f{11}{46}x+\frac{267}{46}\;.$$
The latest example is a further confirmation of what we thought, that is that the line that we find with Algorithm 3 seems to be ``
better'' than others. At this point it would be interesting to study the set of lines in \LL with a different point of view, perhaps less
geometrical and a little bit more statistical, considering that now there's a way to know if the segments are intersected by a straight
line.
\end{document}

%% file: com_seg.tex
\usepackage[utf8]{inputenc}
\usepackage{amsmath}
\usepackage{amssymb}
\usepackage{amsthm}
\usepackage{enumerate}
\usepackage{xspace}
\usepackage{bm}
\usepackage{euscript}
\usepackage{cancel}
\usepackage{mathrsfs}
\usepackage{multido}
\usepackage{pst-plot}
\usepackage{pst-vect}
\newcounter{proposizione}
\newcounter{corollario}
\newcounter{figura}
\newcounter{subsec}
\newrgbcolor{Bluartico}{0.88 0.95 1.00}
\psset{unit=1mm,labelsep=.6mm}
\def\subsec#1{\stepcounter{subsec}\subsection*{\thesubsec{}\ \ #1}}
\def\vvnn#1(#2,#3,#4){\vnode(#2,#3){A_#1}\vnode(#2,#4){B_#1}}
\def\wnn#1(#2,#3,#4){\vnode(5*#2,#3){A_#1}\vnode(5*#2,#4){B_#1}}
\def\f#1#2{\frac{#1}{#2}}
\makeatletter
\def\retta{\@ifnextchar[{\retta@}{\retta@@}}
\def\retta@[#1](#2){{\psset{#1}\retta@@(#2)}}
\def\retta@@(#1){%
  \intersection(5*{0,0|#1})(5*{{1,0}|{#1+{90?#1}}})(55,0)(55,1){@a@}
  \intersection(5*{0,0|#1})(5*{{1,0}|{#1+{90?#1}}})(-2,0)(-2,1){@b@}
  \psline(@a@)(@b@)}
\def\R{\@ifnextchar2{\R@}{\ensuremath{\mathbb{R}}\xspace}}
\def\R@2{\ensuremath{\mathbb{R}^2}\xspace}
\def\axx(#1,#2)(#3,#4){\@ifnextchar({\axx@(#1,#2)(#3,#4)}{\axx@(0,0)(#1,#2)(#3,#4)}}
\def\axx@(#1,#2)(#3,#4)(#5,#6){\axes(#1,#2)(#3,#4)(#5,#6){$x$}{}\uput{2pt}[180](#1,#6){$y$}}
\def\axmb(#1,#2)(#3,#4){\axes(0,0)(#1,#2)(#3,#4){$m$}{}\uput{2pt}[180](0,#4){$b$}}
\makeatother
\def\LL{\ensuremath{\mathscr{L}(\R^2)}\xspace}
\MathCoor


%% file: figura00.tex
\begin{pspicture}[unit=0.8mm](-10,4)(64,55)
  \vvnn1(5,5,35)
  \vvnn2(15,20,50)
  \vvnn3(20,15,40)
  \vvnn4(35,30,45)
  \vvnn5(45,15,50)
  \vvnn6(50,10,60)
  \vnode(5*!-1 62 15 div)x
  \vnode(5*!11 28 3 div)y
  \multido{\n=1+1}6{\psline[linewidth=1pt,dotstyle=o,dotsize=3pt]{o-o}(A_\n)(B_\n)}
  \psline[linecolor=red](x+1 12 div*{x>y})(y)
  \psline[linecolor=red,linestyle=dashed](x+1 12 div*{y>x})(x+14 12 div*{x>y})
\end{pspicture}%

%% file: figura01.tex
\begin{pspicture}[linewidth=0.5pt,framesep=0pt](0,-29)(121,25)
  \intersection(20;70)(20;-70)(-25,25)(25,25)A
  \multido{\n=32+57,\na=145+180,\nu=1+-2}{2}{\rput(\n,0){%
    \multido{\n=-24+4}{13}{\psline[linecolor=lightgray](\n,-25)(\n,25)\psline[linecolor=lightgray](-25,\n)(25,\n)}%
    \axx(-25,-25)(25,25)%
    \uput*{1pt}[117.5](20;117.5){$\mathscr C$}
    \rput*[l](-23.7,-22.3){$\cos\alpha\ifnum\nu=-1<\else>\fi\cos\gamma$}%
    \rput[l](-25,-27){\figura}%
    \psline[linestyle=dashed](20;-70)(20;70)%
    \psline[linestyle=dashed](20;\na|0,0)(20;\na)%
    \point*(20;\na)(20;\na|0,0)(20;\na){$\scriptstyle{\cos\gamma}$}
    \point*(20;70|{\nu*20;70})(20;70|0,0)(20;70|{\nu*20;70}){$\scriptstyle{\cos\alpha}$}
    \point*(0,0)(20;-70)(0,0){$A$}%
    \point*(0,0)(20;\na)(0,0){$C$}%
    \point*(0,0)(20;70)(0,0){$B$}}}%
  \rput(32,0){\psarc[arcsep=4pt]{->}(0,0){20}{145}{-70}%
    \psarc[arcsep=4pt]{->}(0,0){20}{70}{145}%
    \psarc[arcsep=4pt]{->}(0,0){20}{-70}{70}}%
  \rput(89,0){\psarc[arcsep=4pt]{<-}(0,0){20}{-35}{70}%
    \psarc[arcsep=4pt]{<-}(0,0){20}{-70}{-35}%
    \psarc[arcsep=4pt]{<-}(0,0){20}{70}{-70}}%
\end{pspicture}%

%% file: figura02.tex
\begin{pspicture}[c,linewidth=0.6pt,framesep=0pt](-25,-13)(30,30)%
  \multido{\ny=-8+4,\nx=-20+4}{13} {%
    \ifnum\ny>29{}\else\psline[linecolor=lightgray](-22,\ny)(30,\ny)\fi%
    \psline[linecolor=lightgray](\nx,-10)(\nx,30)}%
  \rput[l](-22,-12){\figura}%
  \axx(-22,-10)(30,30)%
  \psline[linewidth=2pt,linecolor=blue](4*1,1)(4*1,4)%
  \point*(4*1,4)(4*1,1)(4*1,4){$B$}%
  \point*(4*1,1)(4*1,4)(4*1,1){$A$}%
  \uput*{2pt}[180](0,4){$\scriptstyle1$}%
  \uput*{2pt}[270](4,0){$\scriptstyle1$}%
  \intersection(0,20)(8,8)(0,-10)(1,-10)A%
  \intersection(0,20)(8,8)(0,30)(1,30)B%
  \psline[linestyle=dashed](A)(B)
  \psline(A+0.1*{A>B})(B+0.1*{B>A})
  \pcline[linestyle=none,offset=-5pt](B)(B+0.25*{B>A})\lput*{:U}{$\scriptscriptstyle E:\,y=-\frac32x+5$}
  \intersection(-16,0)(-8,4)(30,0)(30,1)A%
  \intersection(-16,0)(-8,4)(-22,0)(-22,1)B%
  \psline[linestyle=dashed](A)(B)
  \psline(A+0.1*{A>B})(B+0.1*{B>A})
  \pcline[linestyle=none,offset=5pt](B+0.75*{B>A})(A)\lput*{:U}{$\scriptscriptstyle D:\,y=\frac12x+2$}
  \intersection(28,-8)(-20,24)(30,0)(30,1)A%
  \intersection(28,-8)(-20,24)(-22,0)(-22,1)B%
  \psline[linestyle=dashed](A)(B)
  \psline(A+0.1*{A>B})(B+0.1*{B>A})
  \pcline[linestyle=none,offset=5pt](B)(B+0.25*{B>A})\lput*{:U}{$\scriptscriptstyle F:\,y=-\frac23x+\frac83$}
  \intersection(0,12)(4,12)(30,0)(30,1)A%
  \intersection(0,12)(4,12)(-22,0)(-22,1)B%
  \psline[linestyle=dashed](A)(B)
  \psline(A+0.1*{A>B})(B+0.1*{B>A})
  \pcline[linestyle=none,offset=5pt](B)(B+0.15*{B>A})\lput*{:U}{$\scriptscriptstyle C:\,y=3$}%
\end{pspicture}%

%% file: figura03.tex
\begin{pspicture}[c,linewidth=0.6pt,framesep=0pt](-26,-13)(29,30)%
  \intersection(0,4)(4,0)(0,-10)(1,-10)A%
  \intersection(0,4)(4,0)(0,30)(1,30)B%
  \intersection(0,16)(16,0)(0,-10)(1,-10)C%
  \intersection(0,16)(16,0)(0,30)(1,30)D%
  \pspolygon[linestyle=none,fillstyle=solid,fillcolor=Bluartico](A)(B)(D)(C)
  \multido{\ny=-8+4,\nx=-24+4}{13}{%
    \ifnum\ny>29{}\else\psline[linecolor=lightgray](-26,\ny)(26,\ny)\fi%
    \psline[linecolor=lightgray](\nx,-10)(\nx,30)}
  \axx(-26,-10)(26,30)%
  \uput*{2pt}[202.5](0,4){$\scriptstyle1$}%
  \uput*{2pt}[22.5](0,16){$\scriptstyle4$}%
  \uput*{2pt}[247.5](4,0){$\scriptstyle1$}%
  \psline[linestyle=dashed](A)(B)
  \psline(A+0.1*{A>B})(B+0.1*{B>A})
  \psline[linestyle=dashed](C)(D)
  \psline(C+0.1*{C>D})(D+0.1*{D>C})
  {\psset{fillcolor=Bluartico}
    \uput*{1pt}[140](4*0,3){$\scriptstyle C$}
    \uput*{1pt}[-5](2,8){$\scriptstyle D$}
    \uput*{2pt}[185](-6,20){$\scriptstyle E$}
    \uput*[275](4 3 div*-2,8){$\scriptstyle F$}}
  \psdots[dotstyle=o](4*0,3)(2,8)(-6,20)(4 3 div*-2,8)
\end{pspicture}%

%% file: figura04.tex
\vvnn1(10,5,35)
\vvnn2(25,20,50)
\vvnn3(35,15,40)
\vvnn4(55,30,45)
\vvnn5(65,15,50)
\vvnn6(100,10,60)
\begin{pspicture}[framesep=0pt,c](-5,-6)(110,63)
  \multido{\n=5+5}{12}{\psline[linecolor=lightgray](-5,\n)(110,\n)}
  \multido{\n=5+5}{21}{\psline[linecolor=lightgray](\n,-3)(\n,63)}
  \axx(-5,-3)(110,63)
  \multido{\n=1+1}6{\psline[linewidth=2pt,linecolor=blue](A_\n)(B_\n)}
  \multido{\n=1+1}{5}{\name*(A_\n)(B_\n){$L_\n$}\point*(A_\n)(B_\n)(A_\n){$B_\n(x_\n,b_\n)$}
    \point*(B_\n)(A_\n)(B_\n){$A_\n(x_\n,a_\n)$}}
  \name*(A_6)(B_6){$L_n$}
  \point*(A_6)(B_6)(A_6){$B_n(x_n,b_n)$}\point*(B_6)(A_6)(B_6){$A_n(x_n,a_n)$}
  \rput*(82.5,35){\Huge$\cdots$}
  \rput[l](-5,-5){\figura}
  \name(A_1)(B_1){$L_1$}
\end{pspicture}%

%% file: figura05.tex
\small%
\vvnn1(5,5,35)%
\vvnn2(15,20,50)%
\vvnn3(20,15,40)%
\vvnn4(35,30,45)%
\vvnn5(45,15,50)%
\vvnn6(50,10,60)%
\begin{pspicture}[c,framesep=0pt](-4,-6)(56,66)
  \pspolygon[linestyle=none,fillstyle=solid,fillcolor=Bluartico](A_1)(A_2)(A_3)(A_4)(A_5)(A_6)(B_6)(B_5)(B_4)(B_3)(B_2)(B_1)
  \multido{\n=5+5}{12}{\psline[linecolor=lightgray](-2,\n)(55,\n)}
  \multido{\n=5+5}{10}{\psline[linecolor=lightgray](\n,-2)(\n,65)}
  \axx(-2,-2)(55,65)
  \multido{\n=1+1}6{\psline[linewidth=1pt,linecolor=blue](A_\n)(B_\n)}
  \intersection(A_2)(B_5)(55,0)(55,1)X
  \intersection(A_2)(B_5)(-2,0)(-2,1)W
  \psline[linestyle=dashed,linewidth=0.6pt](X)(W)
  \pcline[linestyle=none,offset=3pt](X+0.08*{X>W})(X)\lput*{:U}{$\bm r$}
  \intersection(A_4)(B_1)(55,0)(55,1)Y
  \intersection(A_4)(B_1)(-2,0)(-2,1)Z
  \psline[linestyle=dashed,linewidth=0.6pt](Y)(Z)
  \pcline[linestyle=none,offset=4pt](Y)(Y+0.08*{Y>Z})\lput*{:D}{$\bm p$}
  \psline[linecolor=red](0.5*{X+Y})(0.5*{W+Z})
  \pcline[linestyle=none,offset=3pt](0.5*{X+Y}+0.08*{{0.5*{X+Y}}>{0.5*{W+Z}}})(0.5*{X+Y})\lput*{:U}{\color{red}$\bm{s_1}$}
  \vpoint*(B_1)(A_1)(A_2)(A_3)(A_4)(A_5)(A_6)(B_6)(B_5)(B_4)(B_3)(B_2)(B_1)(Z)
  \multido{\n=1+1}{10}{\uput*{2pt}[d](5*\n,0){\n}}
  \multido{\n=1+1}{12}{\uput*{2pt}[l](5*0,\n){\n}}
  \uput[r](-4,-5){\figura}
\end{pspicture}%

%% file: figura06.tex
\vvnn3(20,30,40)%
\vvnn4(35,20,45)%
\begin{pspicture}[c,framesep=0pt](-3,-6)(56,66)
  \pspolygon[linestyle=none,fillstyle=solid,fillcolor=Bluartico](A_1)(A_2)(A_3)(A_4)(A_5)(A_6)(B_6)(B_5)(B_4)(B_3)(B_2)(B_1)
  \multido{\n=5+5}{12}{\psline[linecolor=lightgray](-2,\n)(55,\n)}
  \multido{\n=5+5}{10}{\psline[linecolor=lightgray](\n,-2)(\n,65)}
  \axx(-2,-2)(55,65)
  \multido{\n=1+1}6{\psline[linewidth=1pt,linecolor=blue](A_\n)(B_\n)}
  \intersection(A_3)(B_5)(55,0)(55,1)X
  \intersection(A_3)(B_5)(-2,0)(-2,1)W
  \psline[linestyle=dashed,linewidth=0.6pt](X)(W)
  \pcline[linestyle=none,offset=3pt](X+0.08*{X>W})(X)\lput*{:U}{$\bm r$}
  \intersection(A_3)(B_1)(55,0)(55,1)Y
  \intersection(A_3)(B_1)(-2,0)(-2,1)Z
  \psline[linestyle=dashed,linewidth=0.6pt](Y)(Z)
  \pcline[linestyle=none,offset=4pt](Y)(Y+0.08*{Y>Z})\lput*{:D}{$\bm p$}
  \psline[linecolor=red](0.5*{X+Y})(0.5*{W+Z})
  \pcline[linestyle=none,offset=3pt](0.5*{X+Y}+0.08*{{0.5*{X+Y}}>{0.5*{W+Z}}})(0.5*{X+Y})\lput*{:U}{\color{red}$\bm{s_1}$}
  \vpoint*(B_1)(A_1)(A_2)(A_3)(A_4)(A_5)(A_6)(B_6)(B_5)(B_4)(B_3)(B_2)(B_1)(Z)
  \multido{\n=1+1}{10}{\uput*{2pt}[d](5*\n,0){\n}}
  \multido{\n=1+1}{12}{\uput*{2pt}[l](5*0,\n){\n}}
  \uput[r](-4,-5){\figura}
\end{pspicture}%

%% file: figura07.tex
\small%
\vvnn1(5,5,35)%
\vvnn2(15,20,50)%
\vvnn3(20,15,40)%
\vvnn4(35,30,45)%
\vvnn5(45,15,50)%
\vvnn6(50,10,60)%
\begin{pspicture}[c,framesep=0pt](-4,-6)(56,66)
  \pspolygon[linestyle=none,fillstyle=solid,fillcolor=Bluartico](A_1)(A_2)(A_3)(A_4)(A_5)(A_6)(B_6)(B_5)(B_4)(B_3)(B_2)(B_1)
  \multido{\n=5+5}{12}{\psline[linecolor=lightgray](-2,\n)(55,\n)}
  \multido{\n=5+5}{10}{\psline[linecolor=lightgray](\n,-2)(\n,65)}
  \axx(-2,-2)(55,65)
  \multido{\n=1+1}6{\psline[linewidth=1pt,linecolor=blue](A_\n)(B_\n)}
  \intersection(A_2)(B_5)(55,0)(55,1)X
  \intersection(A_2)(B_5)(-2,0)(-2,1)W
  \psline[linestyle=dashed,linewidth=0.6pt](X)(W)
  \pcline[linestyle=none,offset=3pt](X+0.08*{X>W})(X)\lput*{:U}{$\bm r$}
  \intersection(A_4)(B_1)(55,0)(55,1)Y
  \intersection(A_4)(B_1)(-2,0)(-2,1)Q
  \psline[linestyle=dashed,linewidth=0.6pt](Y)(Q)
  \pcline[linestyle=none,offset=4pt](Y)(Y+0.08*{Y>Q})\lput*{:D}{$\bm p$}
  \intersection(A_2)(A_4)(55,0)(55,1)Y
  \intersection(A_2)(A_4)(-2,0)(-2,1)Z
  \psline[linestyle=dashed,linewidth=0.6pt](Y)(Z)
  \intersection(B_4)(B_1)(55,0)(55,1)Y
  \intersection(B_4)(B_1)(-2,0)(-2,1)Z
  \psline[linestyle=dashed,linewidth=0.6pt](Y)(Z)
  \intersection(B_4)(B_5)(55,0)(55,1)Y
  \intersection(B_4)(B_5)(-2,0)(-2,1)Z
  \psline[linestyle=dashed,linewidth=0.6pt](Y)(Z)
  \intersection(5*1,5)(5*137 30 div*0,1)(55,0)(55,1)a
  \intersection(5*1,5)(5*137 30 div*0,1)(-2,0)(-2,1)b
  \psline[linecolor=red](a)(b)
  \pcline[linestyle=none,offset=3pt](a+0.08*{a>b})(a)\lput*{:U}{\color{red}$\bm{s_2}$}
  \vpoint*(B_1)(A_1)(A_2)(A_3)(A_4)(A_5)(A_6)(B_6)(B_5)(B_4)(B_3)(B_2)(B_1)(Q)
  \multido{\n=1+1}{10}{\uput*{2pt}[d](5*\n,0){\n}}
  \multido{\n=1+1}{12}{\uput*{2pt}[l](5*0,\n){\n}}
  \uput[r](-4,-5){\figura}
\end{pspicture}

%% file: figura08.tex
\begin{pspicture}[c,framesep=0pt](-10,-8)(47,64)
  \vnode(21,15){P_1}
  \vnode(-7,43){P_2}
  \vnode(14,40){P_3}
  \vnode(21,33){P_4}
  \vnode(42,06){P_5}
  \pspolygon[fillstyle=solid, fillcolor=Bluartico](P_1)(P_2)(P_3)(P_4)(P_5)
  {\psset{linecolor=lightgray}
    \psline(-7,-3)(-7,63)
    \psline(14,-3)(14,63)
    \psline(21,-3)(21,63)
    \psline(42,-3)(42,63)
    \psline(-10,15)(47,15)
    \psline(-10,6)(47,6)
    \psline(-10,33)(47,33)
    \psline(-10,40)(47,40)
    \psline(-10,43)(47,43)}
  \axmb(-10,-3)(47,63)
  \vpoint*2pt(P_4)(P_3)(P_2)(P_1)
  \vpoint2pt(P_2)(P_1)(P_5)(P_4)(P_3)
  \point[linecolor=red,fillcolor=red](P_5)(1 5 div*{P_1+P_2+P_3+P_4+P_5})(P_5){\color{red}G}
  \uput*[d](-7,0){$-\frac16$}
  \uput*[d](14,0){$\frac13$}
  \uput*[d](21,0){$\frac12$}
  \uput*[d](42,0){$1$}
  \uput*[l](0,6){$1$}
  \uput*[l](0,15){$\frac52$}
  \uput*[l](0,33){$\frac{11}2$}
  \uput[ur](8,48){$\frac{43}6$}
  \uput[dl](-6,39){$\frac{20}3$}
  \psline[linewidth=0.3pt,linestyle=dotted,dotsep=1pt]{->}(9.3,50.1)(0.3,43.3)
  \psline[linewidth=0.3pt,linestyle=dotted,dotsep=1pt]{->}(-7.3,37.05)(-0.3,39.7)
  \uput{0pt}[r](-10,-7){\figura}
\end{pspicture}%

%% file: figura09.tex
\small%
\vvnn1(5,5,35)%
\vvnn2(15,20,50)%
\vvnn3(20,30,40)%
\vvnn4(35,20,45)%
\vvnn5(45,15,50)%
\vvnn6(50,10,60)%
\begin{pspicture}[c,framesep=0pt](-4,-6)(56,66)
  \pspolygon[linestyle=none,fillstyle=solid,fillcolor=Bluartico](A_1)(A_2)(A_3)(A_4)(A_5)(A_6)(B_6)(B_5)(B_4)(B_3)(B_2)(B_1)
  \multido{\n=5+5}{12}{\psline[linecolor=lightgray](-2,\n)(55,\n)}
  \multido{\n=5+5}{10}{\psline[linecolor=lightgray](\n,-2)(\n,65)}
  \axx(-2,-2)(55,65)
  \multido{\n=1+1}6{\psline[linewidth=1pt,linecolor=blue](A_\n)(B_\n)}
  \intersection(A_3)(B_5)(55,0)(55,1)X
  \intersection(A_3)(B_5)(-2,0)(-2,1)W
  \psline[linestyle=dashed,linewidth=0.6pt](X)(W)
  \pcline[linestyle=none,offset=3pt](X+0.08*{X>W})(X)\lput*{:U}{$\bm r$}
  \intersection(A_3)(B_1)(55,0)(55,1)Y
  \intersection(A_3)(B_1)(-2,0)(-2,1)Q
  \psline[linestyle=dashed,linewidth=0.6pt](Y)(Q)
  \pcline[linestyle=none,offset=4pt](Y)(Y+0.08*{Y>Q})\lput*{:D}{$\bm p$}
  \intersection(B_4)(B_1)(55,0)(55,1)Y
  \intersection(B_4)(B_1)(-2,0)(-2,1)Z
  \psline[linestyle=dashed,linewidth=0.6pt](Y)(Z)
  \intersection(B_4)(B_5)(55,0)(55,1)Y
  \intersection(B_4)(B_5)(-2,0)(-2,1)Z
  \psline[linestyle=dashed,linewidth=0.6pt](Y)(Z)
  \intersection(5*-11,2)(5*223 40 div*0,1)(55,0)(55,1)a
  \intersection(5*-11,2)(5*223 40 div*0,1)(-2,0)(-2,1)b
  \psline[linecolor=red](a)(b)
  \pcline[linestyle=none,offset=3pt](a+0.08*{a>b})(a)\lput*{:U}{\color{red}$\bm{s_2}$}
  \vpoint*(B_1)(A_1)(A_2)(A_3)(A_4)(A_5)(A_6)(B_6)(B_5)(B_4)(B_3)(B_2)(B_1)(Q)
  \multido{\n=1+1}{10}{\uput*{2pt}[d](5*\n,0){\n}}
  \multido{\n=1+1}{12}{\uput*{2pt}[l](5*0,\n){\n}}
  \uput[r](-4,-5){\figura}
\end{pspicture}%

%% file: figura10.tex
\begin{pspicture}[c,framesep=0pt](-19,-8)(38,64)
  \vnode(!4 42 mul 5 div 14 6 mul 5 div){P_1}
  \vnode(-14,44){P_2}
  \vnode(14,40){P_3}
  \vnode(21,33){P_4}
  \pspolygon[fillstyle=solid, fillcolor=Bluartico,linearc=0.05](P_1)(P_2)(P_3)(P_4)
  {\psset{linecolor=lightgray}
    \psline(-14,-3)(-14,63)
    \psline(14,-3)(14,63)
    \psline(21,-3)(21,63)
    \psline(P_1|0,-3)(P_1|0,63)
    \psline(-19,0|P_1)(38,0|P_1)
    \psline(-19,33)(38,33)
    \psline(-19,40)(38,40)
    \psline(-19,44)(38,44)}
  \axmb(-19,-3)(38,63)
  \vpoint*2pt(P_2)(P_1)(P_4)(P_3)(P_2)(P_1)
  \point[linecolor=red,fillcolor=red](P_1)(1 4 div*{P_1+P_2+P_3+P_4})(P_1){\color{red}G}
  \uput*[d](-14,0){$-\frac13$}
  \uput*[d](14,0){$\frac13$}
  \uput*[d](21,0){$\frac12$}
  \uput*[d](P_1|0,0){$\frac45$}
  \uput*[l](0,0|P_1){$\frac{14}5$}
  \uput*[l](0,33){$\frac{11}2$}
  \uput[ur](8,48){$\frac{22}3$}
  \uput[dl](-6,39){$\frac{20}3$}
  \psline[linewidth=0.3pt,linestyle=dotted,dotsep=1pt]{->}(9.3,50.1)(0.3,44.3)
  \psline[linewidth=0.3pt,linestyle=dotted,dotsep=1pt]{->}(-7.3,37.05)(-0.3,39.7)
  \uput{0pt}[r](-19,-7){\figura}
\end{pspicture}%

%% file: figura11.tex
\small%
\vvnn1(5,5,35)%
\vvnn2(15,20,50)%
\vvnn3(20,15,40)%
\vvnn4(35,30,45)%
\vvnn5(45,15,50)%
\vvnn6(50,10,60)%
\begin{pspicture}[c,framesep=0pt](-4,-6)(56,66)
  \pspolygon[linestyle=none,fillstyle=solid,fillcolor=Bluartico](A_1)(A_2)(A_3)(A_4)(A_5)(A_6)(B_6)(B_5)(B_4)(B_3)(B_2)(B_1)
  \multido{\n=5+5}{12}{\psline[linecolor=lightgray](-2,\n)(55,\n)}
  \multido{\n=5+5}{10}{\psline[linecolor=lightgray](\n,-2)(\n,65)}
  \axx(-2,-2)(55,65)
  \multido{\n=1+1}6{\psline[linewidth=1pt,linecolor=blue](A_\n)(B_\n)}
  \intersection(5*1,5)(5*137 30 div*0,1)(55,0)(55,1)a
  \intersection(5*1,5)(5*137 30 div*0,1)(-2,0)(-2,1)b
  \psline[linestyle=dashed,linecolor=green](a)(b)
  \pcline[linestyle=none,offset=4pt](a+0.08*{a>b})(a)\lput*{:U}{\color{green}$\bm{s_2}$}
  \intersection(A_2)(B_5)(55,0)(55,1)X
  \intersection(A_2)(B_5)(-2,0)(-2,1)W
  \intersection(A_4)(B_1)(55,0)(55,1)Y
  \intersection(A_4)(B_1)(-2,0)(-2,1)Z
  \psline[linestyle=dashed,linecolor=cyan](0.5*{X+Y})(0.5*{W+Z})
  \pcline[linestyle=none,offset=4pt](0.5*{X+Y})(0.5*{X+Y}+0.08*{{0.5*{X+Y}}>{0.5*{W+Z}}})\lput*{:D}{\color{cyan}$\bm{s_1}$}
  \intersection(!0 107 24 div 5 mul)(!5 117 24 div 5 mul)(55,0)(55,1)X
  \intersection(!0 107 24 div 5 mul)(!5 117 24 div 5 mul)(-2,0)(-2,1)W
  \psline[linecolor=red](X)(W)
  \pcline[linestyle=none,offset=1pt](X)(X+0.08*{X>W})\lput*{:D}{\color{red}$\bm{s_3}$}
  \vpoint*(B_1)(A_1)(A_2)(A_3)(A_4)(A_5)(A_6)(B_6)(B_5)(B_4)(B_3)(B_2)(B_1)(A_1)
  \multido{\n=1+1}{10}{\uput*{2pt}[d](5*\n,0){\n}}
  \multido{\n=1+1}{12}{\uput*{2pt}[l](5*0,\n){\n}}
  \uput[r](-4,-5){\figura}
\end{pspicture}%

%% file: figura12.tex
\vvnn1(5,5,35)%
\vvnn2(15,20,50)%
\vvnn3(20,30,40)%
\vvnn4(35,20,45)%
\vvnn5(45,15,50)%
\vvnn6(50,10,60)%
\begin{pspicture}[c,framesep=0pt](-4,-6)(56,66)
  \pspolygon[linestyle=none,fillstyle=solid,fillcolor=Bluartico](A_1)(A_2)(A_3)(A_4)(A_5)(A_6)(B_6)(B_5)(B_4)(B_3)(B_2)(B_1)
  \multido{\n=5+5}{12}{\psline[linecolor=lightgray](-2,\n)(55,\n)}
  \multido{\n=5+5}{10}{\psline[linecolor=lightgray](\n,-2)(\n,65)}
  \axx(-2,-2)(55,65)
  \multido{\n=1+1}6{\psline[linewidth=1pt,linecolor=blue](A_\n)(B_\n)}
  \intersection(5*-11,2)(5*223 40 div*0,1)(55,0)(55,1)a
  \intersection(5*-11,2)(5*223 40 div*0,1)(-2,0)(-2,1)b
  \psline[linestyle=dashed,linecolor=green](a)(b)
  \pcline[linestyle=none,offset=4pt](a+0.08*{a>b})(a)\lput*{:U}{\color{green}$\bm{s_2}$}
  \intersection(A_3)(B_5)(55,0)(55,1)X
  \intersection(A_3)(B_5)(-2,0)(-2,1)W
  \intersection(A_3)(B_1)(55,0)(55,1)Y
  \intersection(A_3)(B_1)(-2,0)(-2,1)Z
  \psline[linestyle=dashed,linecolor=cyan](0.5*{X+Y})(0.5*{W+Z})
  \pcline[linestyle=none,offset=4pt](0.5*{X+Y})(0.5*{X+Y}+0.08*{{0.5*{X+Y}}>{0.5*{W+Z}}})\lput*{:D}{\color{cyan}$\bm{s_1}$}
  \intersection(!0 11873 2130 div 5 mul)(!5 12466 2130 div 5 mul)(55,0)(55,1)X
  \intersection(!0 11873 2130 div 5 mul)(!5 12466 2130 div 5 mul)(-2,0)(-2,1)W
  \psline[linecolor=red](X)(W)
  \pcline[linestyle=none,offset=1pt](X)(X+0.08*{X>W})\lput*{:D}{\color{red}$\bm{s_3}$}
  \vpoint*(B_1)(A_1)(A_2)(A_3)(A_4)(A_5)(A_6)(B_6)(B_5)(B_4)(B_3)(B_2)(B_1)(A_1)
  \multido{\n=1+1}{10}{\uput*{2pt}[d](5*\n,0){\n}}
  \multido{\n=1+1}{12}{\uput*{2pt}[l](5*0,\n){\n}}
  \uput[r](-4,-5){\figura}
\end{pspicture}%

%% file: figura13.tex
\small%
\wnn1(1,1,11)%
\wnn2(2,2,12)%
\wnn3(4,3,12)%
\wnn4(6,2,11)%
\wnn5(8,2,12)%
\wnn6(9,3,13)%
\begin{pspicture}[c,framesep=0pt](-4,-6)(56,71)\small
  {\psset{linestyle=none,fillstyle=solid,fillcolor=Bluartico}
    \pspolygon(A_1)(A_2)(A_3)(A_4)(A_5)(A_6)(B_6)(B_5)(B_4)(B_3)(B_2)(B_1)}
  \multido{\n=5+5}{13}{\psline[linecolor=lightgray](-2,\n)(55,\n)}
  \multido{\n=5+5}{10}{\psline[linecolor=lightgray](\n,-2)(\n,70)}
  \axx(-2,-2)(55,70)
  \multido{\n=1+1}6{\psline[linewidth=1pt,linecolor=blue](A_\n)(B_\n)}
  \retta[linecolor=cyan](!1 4 div 23 4 div)
  \pcline[linestyle=none,offset=4pt](@a@+0.08*{@a@>@b@})(@a@)\lput*{:U}{\color{cyan}$\bm{s_1}$}
  \retta[linecolor=green](!7 16 div 77 16 div)
  \pcline[linestyle=none,offset=4pt](@a@+0.08*{@a@>@b@})(@a@)\lput*{:U}{\color{green}$\bm{s_2}$}
  \retta[linecolor=red](!11 46 div 267 46 div)
  \pcline[linestyle=none,offset=4pt](@a@)(@a@+0.08*{@a@>@b@})\lput*{:D}{\color{red}$\bm{s_3}$}
  \vpoint*(A_1)(A_2)(A_3)(A_4)(A_5)(A_6)(B_6)(B_5)(B_4)(B_3)(B_2)(B_1)(A_1)(A_2)
  \multido{\n=1+1}{10}{\uput*{2pt}[d](5*\n,0){\n}}
  \multido{\n=1+1}{13}{\uput*{2pt}[l](5*0,\n){\n}}
  \uput[r](-4,-5){\figura}
\end{pspicture}%

%% file: figura14.tex
\begin{pspicture}[c,framesep=0pt](-23,-10)(35,67)
  \vnode(21,0){P_1}
  \vnode(10.5,5){P_2}
  \vnode(0,15){P_3}
  \vnode(-21,60){P_4}
  \vnode(0,55){P_5}
  \vnode(10.5,40){P_6}
  \vnode(21,20){P_7}
  \vnode(31.5,-2.5){P_8}
  \pspolygon[fillstyle=solid, fillcolor=Bluartico,linearc=0.05](P_1)(P_2)(P_3)(P_4)(P_5)(P_6)(P_7)(P_8)
  {\psset{linecolor=lightgray}
    \psline(P_1|0,-6)(P_1|0,66)
    \psline(P_2|0,-6)(P_2|0,66)
    \psline(P_4|0,-6)(P_4|0,66)
    \psline(P_8|0,-6)(P_8|0,66)
    \psline(-23,0|P_2)(35,0|P_2)
    \psline(-23,0|P_8)(35,0|P_8)
    \psline(-23,0|P_3)(35,0|P_3)
    \psline(-23,0|P_4)(35,0|P_4)
    \psline(-23,0|P_5)(35,0|P_5)
    \psline(-23,0|P_6)(35,0|P_6)
    \psline(-23,0|P_7)(35,0|P_7)}
  \axmb(-23,-6)(35,66)
  \vpoint*2pt(P_2)(P_1)(P_8)(P_7)(P_6)(P_5)(P_4)(P_3)(P_2)(P_1)
  \uput*[u](P_8|0,0){$\frac32$}
  \uput*[d](P_6|0,0){$\frac12$}
  \uput*[d](P_4|0,0){$-1$}
  \uput*[l](0,0|P_4){$12$}
  {\psset{fillcolor=Bluartico}
    \uput*[u](P_7|0,0){$1$}
    \uput*[15](0,0|P_3){$3$}
    \uput*[r](0,0|P_7){$4$}
    \uput*[dl](0,0|P_5){$11$}
    \uput*[l](0,0|P_6){$8$}}
  \uput*[l](0,0|P_2){$1$}
  \uput*[l](0,0|P_8){$-\frac12$}
  \uput*[45](0,0){$O$}
  \uput{0pt}[r](-23,-9){\figura}
\end{pspicture}%